\documentclass[final,1p,incs,times,english]{amsart}
\usepackage{color,latexsym,amsmath,amsthm,amsfonts,amssymb,cite}
\usepackage[utf8]{inputenc}
\usepackage{layout,verbatim,setspace,xcolor}
\usepackage{times}
\usepackage{enumerate}

\def\be{\begin{equation}}
\def\ee{\end{equation}}
\def\ba*{\begin{eqnarray*}}
\def\ea*{\end{eqnarray*}}
\newcommand{\vG}{\varGamma}
\newcommand{\ve}{\varepsilon}

\newcommand{\vS}{\varSigma}

\newcommand{\N}{\mathbb{N}}
\newcommand{\R}{\mathbb{R}}
\newcommand{\mcL}{{\mathcal L}}

\newcommand{\mcI}{{\mathcal I}}

\newcommand{\mcH}{\mathcal H}

\newcommand{\vPh}{\varPhi}

\def\id{$[0,1]$}

\newtheorem{thm}{Theorem}[section]
\newtheorem{deff}[thm]{Definition}
\newtheorem{lem}[thm]{Lemma}

\newtheorem{rem}[thm]{Remark}
\newtheorem{prop}[thm]{Proposition}

\newtheorem{que}[thm]{Question}
\newtheorem {exa}[thm]{Example}

\begin{document}
\hyphenation{pro-per-ties} \hyphenation{cha-rac-te-ri-za-tions}
\hyphenation{stron-gly} \hyphenation{pro-per-ties}
\hyphenation{Pro-per-ty} \hyphenation{si-mi-lar}
\hyphenation{re-gu-la-ri-ty} \hyphenation{ine-qua-li-ty}
\hyphenation{non-over-lapping} \hyphenation{glo-bal-ly}
\hyphenation{sub-in-ter-vals} \hyphenation{se-quen-ce}
\hyphenation{u-ni-for-mly} \hyphenation{ge-ne-ra-li-zed}
\hyphenation{con-ti-nu-ous} \hyphenation{glo-bal-ly}
\hyphenation{sub-in-ter-vals} \hyphenation{se-quen-ce}
\hyphenation{u-ni-for-mly} \hyphenation{ge-ne-ra-li-zed}
\hyphenation{ter-mi-no-lo-gy} \hyphenation{re-fe-ren-ce}
\hyphenation{Theo-rem} \hyphenation{sub-in-ter-vals}
\hyphenation{dif-fe-ren-tia-ble} \hyphenation{u-ni-for-mly}
\hyphenation{e-qui-va-lent} \hyphenation{no-ti-cing}
\hyphenation{ge-ne-ra-ted}

{\em \small Dedicated to Prof.  Giulianella Coletti with deep esteem and sincere friendship on the occasion of Her 70th birthday\\}
\title{ Multifunctions determined by integrable functions}
\thanks{This research was partially supported by Grant
"Metodi di analisi reale per l'appros\-simazione attraverso operatori discreti e applicazioni”, (2019)  of GNAMPA -- INDAM (Italy), by University of Perugia --  Fondo Ricerca di Base 2018
 and by University of Palermo.}
\author[add1]{D. Candeloro}\email{domenico.candeloro@unipg.it}
\author[add2]{L. Di Piazza} \email{luisa.dipiazza@unipa.it}
\author[add3]{ K. Musia{\l}} \email{musial@math.uni.wroc.pl}
\author[add1]{A. R. Sambucini
 \email{anna.sambucini@unipg.it}}
\address[add1]{Department of Mathematics and Computer Sciences - 06123 Perugia (Italy)}
\address[add2]{Department of Mathematics, University of Palermo,  Via Archirafi 34, 90123 Palermo (Italy)}
\address[add3]{Institut of Mathematics, Wroc{\l}aw University,  Pl. Grunwaldzki  2/4, 50-384 Wroc{\l}aw (Poland)}
\maketitle
\begin{abstract}
 Integral properties of multifunctions determined by  vector valued functions are presented.   
 Such multifunctions  quite often serve as examples and counterexamples.
In particular it can be observed that the properties of being integrable in the sense of Bochner, McShane or Birkhoff can be transferred  to the generated multifunction while Henstock integrability does not guarantee it.\\
\end{abstract}

\noindent
{ \bf keyword}:
Positive multifunction, gauge integral,   selection, multifunction determined by a function, measure theory.
\\

\noindent
\subjclass[MSC 2010]  28B20,  26E25, 26A39,
 28B05,  46G10, 54C60, 54C65


\section*{Introduction}
 The theory of multifunctions  is an important
field of investigations as theoretical applications and it also allows  to take into account the multiplicity of possible choices in a lot of situations ranging
from Optimal Control to Economic Theory.
In recent years, particular attention has been paid to the study of interval-valued multi-functions because they have a vast range of applications that varies from the representation of uncertainty, to interval-probability, to martingales of multivalued functions (see for example \cite{pap,pap-iran} or \cite{gav2014} and references therein).
In particular the use of intervals to represent uncertainty in the area of decision  and information theory
 has been suggested by several authors.
  At the same time, positive interval-valued multifunctions have also played an important role in applications and they arise quite naturally, for example,  in the context of fractal image coding, as shown in \cite{latorre} or in differential inclusions  (see for example \cite{ci1,dms,dms1}).
  
In the recent literature, several  methods of integration for
Banach-space valued  functions and multifunctions have been studied, based on  various possible constructions of the
Lebesgue integral and on the definition of the Kurzweil-Henstock  integral for real valued functions. 
This is due to the fact that even in case of real valued functions, the Lebesgue integral is not the suitable  tool, for example, if we want integrate a derivative: to this aim it needs to use the Kurzweil-Henstock  integral.
Moreover the study of non-additive set functions and set multifunctions has recently
received a special attention, because of its applications in statistics, biology, theory of games and economics. For this purpose the Choquet integral is a powerful tool (see for example \cite{sun}): 
as an example we recall that
in Dempster-Shafer's mathematical theory of evidence
the Belief interval of an event $A$ is the range defined  by the minimum and the maximum
values which could be assigned to $A: [Bel(A); Pl(A)]$ (see also \cite{capot,coletti0,coletti,pva}) where $Bel$ and $Pl$ are the 
Belief and Plausibility functions  that are defined by a basic probability assignment $m$; so to each event $A$ we can assign an interval valued multimeasure.
While,
in the case of vector valued functions, elementary classical examples show that
 the Bochner integral  is highly restrictive; in fact it integrates few functions: for example the function $f: [0,1] \rightarrow l_{\infty}([0,1])$, defined by $f(t)=\chi_{[0,t]}$ is not Bochner  measurable.
By generalizing in some sense the characterization of the Lebesgue primitives for real valued functions we have the Pettis integral for vector valued functions. Instead, a generalization of  the  Lebesgue integral's definition
by using Riemann sums, produces the McShane  and  the Birkhoff integral for vector
valued functions. If the Banach space is separable, then the Pettis,  McShane  and   Birkhoff integrals concide; but for more general Banach spaces they are in general different (see also \cite{DP,r2008}).
In vector spaces it is also possible  to give a version of Choquet integral, also  combining it  with Pettis integral (cf. \cite{sun,CMeS,ar,park}), this is a new line of research which seems very interesting but which needs further study, also in light of the results contained in \cite{am}.
 
 For this reason in the present paper we consider different integrals for multifunctions $G$  determined by  vector valued functions.
What we want to do in this work is to extend to a Banach space $X$  the interval valued multifunction that normally has values in the compact and convex subsets of $\mathbb{R}$
 (namely $F: [0,1] \to ck(\mathbb{R})$) so that it is also positive: that is we introduce the multifunctions $G$ determined by a vector function $g$:
$G(t)={\rm conv}\{0,g(t)\}$  and we study the properties that are inherited from $g$ from the point of view of the  integrability. 

A  study of such kind of multifunctions  was started  in \cite{monat2019} where their properties  were examined with respect to ``scalarly defined integral'' as Pettis, Henstock-Kurzweil-Pettis, Denjoy-Pettis integrals. In the present article we want to examine  the properties that are inherited from the
 "gauge integrals": Henstock, McShane, Birkhoff and variational integrals. We remember that
the construction of the gauge integrals is very similar to that of the Riemann one, but with one crucial difference: instead of using a mesh  to measure the fineness of a tagged partition,  a gauge is considered, which need not be uniform in the integration domain, see for example \cite{bcs2015,bms, BS2011,ccgs,ncm, cs2015,dp,dp-ma,ckr1,kal}.

The paper is organized as follows: in Section \ref{due} the basic concepts and terminology are introduced in order to define the various type of integrability that are studied.
 In Section \ref{tre} we  study properties of multifunctions $G$ generated by  functions $g$ integrable with respect to gauge integrals. The main results of this section are that a determined multifunction $G$ is Bochner, McShane or  Birkhoff integrable
if and only if $g$ is integrable in the same way  (Proposition  \ref{T4} and Theorem \ref{glim}).
Henstock integrability of $g$ does not guarantee Henstock integrability of $G$ and we do not know whether variational Henstock integrability of $g$ yields the same for $G$.  Only a partial result is obtained (Proposition \ref{partial}); the general case remains an open question.
At the end of this section examples are given in order to show that the results contained in
\cite[Theorem 4.2]{monat2019}
are not ensured if the multifunctions are not necessarily Henstock or $\mathcal{H}$ integrable.
 In our future works we shall investigate the relationships among the set valued
 integrals introduced and other set-valued integrals like that of Choquet and
Sugeno, in order to  have other applications of  the  obtained results.
\section{Definitions, terminology}\label{due}
Throughout the paper   $X$ denotes a Banach space with its dual $X^*$, while $B_X$ is its
 closed unit ball.
 The symbols
$ck(X), (cwk(X))$  denote
the families of all
non empty, convex and compact (weakly compact) subsets  of $X$.  For every $C \subset X$ the $s( \cdot, C)$ denotes the
{\it support function of the set}\,  $C$  and is
defined on $X^*$ by $s(x^*, C) := \sup \{ \langle x^*,x \rangle \colon  \ x
\in C\}$, for each $x^* \in X^*$. $|C|:=\sup\{\|x\|: x\in{C}\}$  and $d_H$ is the Hausdorff metric on the hyperspace $ck(X)$.
 The symbol $\|\cdot\|_{\infty}$ denotes the  sup norm as usual.\,  All functions investigated  are defined on the interval $[0,1]$
 endowed with Lebesgue measure $\lambda$.  $\mcI$ is the collection of all closed subintervals  $I$ of the  interval $[0,1]$, and  with the symbol $|I|$ we mean its $\lambda$-measure.
\\
$G\colon [0,1] \to 2^X \setminus \{\emptyset \}$ is a {\it positive} multifunction if $s(x^*,G)\geq 0$ a.e. for each $x^*\in{X^*}$ separately.
$G:[0,1]\to{ck(X)}$ is {\it  determined by a function} $g:[0,1]\to{X}$ if $G(t)={\rm conv}\{0,g(t)\}$ for every $t\in[0,1]$; obviously  determined multifunctions $G$ are positive.\\
 A function $f:[0,1]\to X$ is called a {\it selection of} $G$ if $f(t) \in G(t)$,   for every $t\in  [0,1]$.\\
A multifunction $G\colon [0,1]\to 2^X \setminus \{\emptyset \}$ is  {\it simple} if it is measurable and has only a finite number of values.
$G\colon[0,1]\to c(X)$ is  {\it scalarly measurable} if for every $ x{}^* \in  X{}^*$, the map $s(x{}^*,G(\cdot))$ is measurable.
$G \colon [0,1]\to{ck(X)}$ is said to be {\it Bochner measurable}  if there exists a sequence of simple multifunctions $G_n \colon [0,1] \to ck(X)$ such that
for almost all $t \in [0,1]$ it is
$\lim_{n\rightarrow \infty}d_H(G_n(t),G(t))=0$.\\
A map $M\colon   \vS \to ck(X)$ is  {\it additive}, if $M(A\cup{B})=M(A) + {M(B)}$ for every  $A, B$  in the $\sigma$-algebra $  \vS$ such that $A \cap B = \emptyset$.  $M$  is called a {\it  multimeasure}  if
$s(x^*,M(\cdot))$ is a finite measure, for every $x^*\in{X^*}$; such $M$ is also
 {\it countably additive} in
the Hausdorff metric
 (in this case the name {\it h-multimeasure} is used).
\\
We consider here gauge $ck(X)$-valued integrals
 (Birkhoff, McShane, Henstock, $\mcH$ and variationally Henstock).

  \begin{deff} \rm
    A multifunction $G:[0,1]\to ck(X)$ is said to be {\it Birkhoff integrable} on $[0,1]$,
   if there exists a set
$\vPh{}_{G}([0,1]) \in ck(X)$
 with the following property: for every $\varepsilon > 0$  there is a countable
partition $P_0$  of $[0,1]$  in $\Sigma$ such that for every countable partition $P = (A{}_n){}_n$
 of $[0,1]$  in $\Sigma$
finer than $P_0$ and any choice
 $T = \{t_n: t_n \in  A_n\,,n\in\N\}$, the series 
$\sum_n\lambda(A{}_n) G(t{}_n)$
 is unconditionally
convergent (in the sense of the Hausdorff metric) and
\begin{eqnarray}\label{e14-a}
d{}_H \biggl(\vPh{}_{G}([0,1]),\sum_n G(t{}_n) \lambda(A{}_n)\biggr)<\ve\,.
\end{eqnarray}
\end{deff}

We recall  that  a {\it partition} ${\mathcal P}$ {\it in } $[0,1]$ is a collection of pairs $\{(I{}_1,t{}_1),$ $ \dots,(I{}_p,t{}_p) \}$,
where $I{}_1,\dots,I{}_p$ are nonoverlapping subintervals of $[0,1]$, and $t{}_i$ is a point of $[0,1]$, $i=1,\dots, p$.
If $\cup^p_{i=1} I{}_i=[0,1]$, then  ${\mathcal P}$ is called {\it a partition of} $[0,1]$. If   $t_i \in I{}_i$, $i=1,\dots,p$,   we say that  ${\mathcal P}$ is  a {\it Perron partition}.

 A  {\it gauge} on $[0,1]$ is a positive function on $[0,1]$. Given a gauge $\delta$ on $[0,1]$,
we say that the  partition ${\mathcal P}$ is $\delta$-{\it fine} if
$I{}_i\subset(t{}_i-\delta(t{}_i),t{}_i+\delta(t{}_i))$, $i=1,\dots,p$.\\

    \begin{deff} \rm
 A multifunction $G:[0,1]\to ck(X)$ is said to be {\it Henstock} (resp. {\it McShane})
  {\it  integrable} on $[0,1]$,  if there exists   $\vPh{}_{G}([0,1]) \in ck(X)$
    with the property that for every $\varepsilon > 0$ there exists a gauge $\delta: [0,1] \to \mathbb{R}^+$
such that for each $\delta$-fine  Perron partition $\{(I{}_1,t{}_1), \dots,(I{}_p,t{}_p)\}$
(resp. partition)  of $[0,1]$ it is
\begin{eqnarray}\label{e14}
d{}_H \biggl(\vPh{}_{G}([0,1]),\sum_{i=1}^p G(t{}_i)|I{}_i|\biggr)<\ve\,.
\end{eqnarray}
If only measurable gauges are taken into account, then we have the definition of $\mcH$ (resp. Birkhoff) integrability. In fact  as showed in \cite[Remark 1]{nara} the Birkhoff integrability also can be seen as a gauge integrability (this has been further highlighted in \cite{ncm}). 
\end{deff}
\begin{deff}\rm
   A multifunction $G:[0,1]\to ck(X)$ is said to be {\it variationally Henstock}
   integrable,
   if there exists
   a multimeasure  $\vPh_{G}: {\mathcal I} \to {ck(X)}$
   such that:
   for every $\ve>0$ there exists a gauge $\delta$
   on $[0,1]$ such that for each $\delta$-fine Perron partition
   $\{(I_1,t_1), \dots,(I_p,t_p)\}$ in  $[0,1]$
it is
\begin{eqnarray}\label{aa}
\sum_{j=1}^pd_H \left(\vPh_{G}(I_j),G(t_j))|I_j|\right)<\ve\,.
\end{eqnarray}
 The set  multifunction  $\vPh_{G}$  is the {\it variational Henstock}  {\it primitive} of $G$.
\end{deff}
Finally  ${\mathcal{S}}_H(G)\;[{\mathcal{S}}_{MS}(G)\,,\,
   \mathcal{S}_P(G)\,,\,  {\mathcal{S}}_B(G)\,,\,  {\mathcal{S}}_{vH}(G)\,,\;...]$ denotes the family of all scalarly measurable selections of $G$
    that are Henstock [McShane, Pettis,  Birkhoff, variationally Henstock, ...] integrable.
Definitions and  properties
unexplained in this paper can be found in \cite{CV,cdpms2016a,dp,mu15}.
\\
We recall also that the
 Ra{\aa}dstr\"{o}m embedding $i:ck(X)\to l_{\infty}(B_{X^*})$ defined   by $i(A):=s(\cdot, A)$ is a
 useful tool to approach the $ck(X)$-valued multifunctions
  (see, for example, 
\cite[Theorem II-19]{CV}) or \cite[Theorem 5.7]{labu}). This embedding  $i$  fulfils the following properties:
\begin{itemize}
\item[$i_1$)] $i(\alpha A \,  + \,  \beta C) = \alpha i(A) + \beta i(C)$ for every $A,C\in  ck(X),\alpha, \beta \in  \mathbb{R}{}^+$;  (the symbol $+$ is the Minkowski addition)
\item[$i_2$)] $d_H(A,C)=\|i(A)-i(C)\|_{\infty},\quad A,C\in  ck(X)$;
\item[$i_3$)] $i(ck(X))$  is a normed closed cone in the space
$l_{\infty}(B_{X^*})$;
\item[$i_4$)] $i(\overline{co}(A \cup B) )= \max \{i(A), i(C)\}$ for all $A,C \in ck(X)$.
\end{itemize}

\section{Multifunctions determined by functions.}\label{tre}	
Now we are going to consider a particular family
among positive multifunctions: those
 that are determined by  integrable functions. Transferring properties from $g$ to $G$ is more complicated than for scalarly defined integrals studied in \cite{monat2019}.
We begin with the following  fact that needs only a simple calculation:
\begin{lem}\label{L2}
If $g:[0,1]\to X$ and $G:={\rm conv}\{0,g(t)\}$,  then
$d_H(G(t),G(t'))\leq \|g(t)-g(t')\|,$
for all $t,t'\in [0,1].$
\end{lem}
\begin{prop}\label{T4}
If $G$ is determined by a strongly measurable $g$, then it is Bochner integrable (that is $G$ is Bochner measurable and  integrably bounded) if and only if $g$ is Bochner integrable.
\end{prop}
\begin{proof} It is easy to see that $i\circ{G}$ satisfies the Lusin property, and therefore it is strongly measurable. Moreover, we have  $\|i(G(t))\|=|G(t)|=\|g(t)\|$ for all $t$. So,  $g$ is Bochner integrable if and only if the mapping $t\mapsto|G(t)|$ is integrable, i.e. $G$ is Bochner integrable if and only if $g$ is Bochner integrable.
\end{proof}
Moreover in \cite[Proposition 3.8]{monat2019} it has been proven:

\begin{prop}\label{p12}
If $G$ is determined by a scalarly  measurable $g$, then it is Pettis integrable in $cwk(X)$ if and only if $g$ is Pettis integrable.
\end{prop}

For the intermediate integrals between Bochner and Pettis we have first:
\begin{lem}\label{L1}
If $g:[0,1] \to{X}$ is McShane (Birkhoff) integrable and $\alpha:[0,1] \to\R$ is a bounded measurable function, then $\alpha g$ is respectively McShane (Birkhoff) integrable.
\end{lem}
\begin{proof}
Let $(\alpha_n)_n$ be a uniformly bounded sequence of simple functions on \id\  that is uniformly convergent to $\alpha$. For each $n\in\N$ let $\nu_n$ be the indefinite McShane integral of $\alpha_n \,g$ and let $\nu_{\alpha{g}}$ be the indefinite Pettis integral of $\alpha\,g$. We have the following relations:
$
\lim_n\|\alpha_n(t)g(t)-\alpha(t)g(t)\|=0\quad\mbox{for every}\;t\in [0,1]
$
and, by the Lebesgue Dominated Convergence Theorem,
$
\lim_n\langle{x^*,\nu_n(E)}\rangle=\langle{x^*,\nu_{\alpha{g}}(E)}\rangle\quad\mbox{for every}\;x^*\in{X^*}\;\mbox{and}\;E\in\mcL.
$
According to \cite[Theorem 2I]{FM}  $\alpha g$ is McShane integrable.
\\
If $g$ is Birkhoff  integrable, then each function $\alpha_n g$ is integrable in the same way. The Birkhoff integrability of $\alpha g$ follows from \cite[Theorem 4]{BP}.
\end{proof}

We need also the following result:
\begin{prop}\label{T2}
{\rm (\cite[Corollary 1.5]{mu15}) }
If a multifunction  $G:[0,1] \to{cwk(X)}$ (resp. $ck(X)$)  is  Pettis
integrable in $cwk(X)$  (resp. $ck(X)$), and $f$ is a scalarly measurable selection of $G$, then
 $f$ is  Pettis
 integrable.
\end{prop}

\begin{prop}\label{p6}
Assume that
$G$ is determined by $g$ and it is
 McShane (Birkhoff)  integrable. Then also $g$ is McShane (Birkhoff) integrable. If $G$ is variationally Henstock integrable, then $g$ is variationally Henstock and Pettis integrable.
\end{prop}
\begin{proof}
In order to prove that $g$ is  McShane (Birkhoff) integrable, given a partition $$\{(I_1,t_1),\ldots,(I_p,t_p)\}\,\, \mbox{  of    }\, [0,1],$$ we have to evaluate the number
\begin{eqnarray}\label{e16}\notag
\sup_{\|x^*\|\leq 1}\left|\sum_is(x^*,G(t_i))|I_i|-\int_0^1s(x^*, G(t))\,dt\right|=d_H\left(\sum_i G(t_i)|I_i|,\int_0^1 G(t)\,dt\right)
\end{eqnarray}
Since
$
s(x^*, G(t))=\langle{x^*,g(t)}\rangle^+
$, we have
$$
d_H\left(\sum_i G(t_i)|I_i|,\int_0^1 G(t)\,dt\right)=\sup_{\|x^*\|\leq 1}\left|\sum_i\langle{x^*,g(t_i)\rangle}^+|I_i|-\int_0^1\langle{x^*,g(t)}\rangle^+\,dt\right|
$$
and so, if $G$ is McShane (Birkhoff) integrable, then the family $\{\langle{x^*,g}\rangle^+\colon \|x^*\|\leq 1\}$ is McShane (Birkhoff) equiintegrable.
\\
Replacing $G$ by $-G$, we obtain McShane integrability of $-G$, what yields the equiintegrability of  the family $\{\langle{x^*,-g}\rangle^+\colon \|x^*\|\leq 1\}$. But $\langle{x^*,-g}\rangle^+=\langle{x^*,g}\rangle^-$. Consequently, if $\delta_{G}$ is chosen for $\{\langle{x^*,g}\rangle^+\colon \|x^*\|\leq 1\}$ and $\delta_{-G}$ for $\{\langle{x^*,g}\rangle^-\colon \|x^*\|\leq 1\}$, then $\delta = \min \{ \delta_{G}, \delta_{-G} \}$ is a proper gauge for $g$. Clearly, if $\delta_{G}$ and $\delta_{-G}$ are measurable, then also $\delta$ is measurable. Thus, if $G$ is McShane (Birkhoff) integrable, then also  $g$ is McShane (Birkhoff) integrable.\\
Assume now that $G$ is variationally Henstock integrable. It follows from
\cite[Corollary 3.7]{cdpms2016b} and \cite[Theorem 4.3]{cdpms2016}
that $G$ is Pettis integrable and hence (cf. Proposition \ref{T2}) also $g$ is Pettis integrable.  Then
\begin{eqnarray*}
\sum_id_H\left(G(t_i)|I_i|,(vH)\int_{I_i} G\right)
&=&\sum_i\sup_{\|x^*\|\leq 1}\left|\langle{x^*,g(t_i)\rangle}^+|I_i|-\int_{I_i}\langle{x^*,g(t)}\rangle^+\,dt\right|\\
&=&\sum_i\sup_{\|x^*\|\leq 1}\left|\int_{I_i}\langle{x^*,g(t)}\rangle^-\,dt-\langle{x^*,g(t_i)\rangle}^-|I_i|\right|
\end{eqnarray*}
It follows that
$$
2\sum_id_H\left(G(t_i)|I_i|,\int_{I_i} G(t)\,dt\right)\geq \sum_i\biggl\|g(t_i)|I_i|-\int_{I_i}g(t)\,dt\biggr\|
$$
and that proves the vH-integrability of $g$.
\end{proof}

\begin{thm}\label{glim}
If $g:[0,1]\to X$, then the multifunction  $G$ determined by $g$ is McShane (Birkhoff) integrable if and only if $g$ is integrable in the same way.
\end{thm}
\begin{proof} The ``only if'' part is contained in Proposition \ref{p6}. \\
Assume now that $g$ is integrable in a proper way. We know that  $g$ is Pettis integrable (see \cite{f1995}) and $G$ is Pettis integrable in $cwk(X)$ (see \cite[Theorem 2.6]{mu15} or Proposition \ref{p12}). In particular, all scalarly measurable selections of $G$ are Pettis  integrable.
It is also clear that the set
$$IS_{G}:=\biggl\{(P)\int_0^1s(t)g(t)dt: s\in \mathcal{M}\biggr\},$$
where $\mathcal{M}$ denotes the set of all measurable functions $\varphi:[0,1]\to [0,1]$, is weakly bounded (hence norm bounded) and convex.
We shall prove that $IS_{G}$ is a weakly compact set being the McShane (Birkhoff) integral of $G$ on $[0,1]$.
\\
In order to prove the weak compactness of $IS_{G}$ take an arbitrary sequence $\{s_n\colon s_n\in\mathcal M\,,n\in\N\}$. Since the set $\{s_n\colon s_n\in\mathcal M\,,
n\in\N\}$ is $L_{\infty}[0,1]$-bounded in $L_1[0,1]$, it is weakly relatively compact in $L_1[0,1]$. Assume for simplicity that $s_n\to{s}$ weakly in  $L_1[0,1]$, where $s\geq 0$ everywhere. It is clear that one may assume that $s\in{\mathcal M}$. It follows from the Lebesgue Dominated Convergence Theorem that
$$
\lim_n\int_0^1s_nh\,d\lambda =  \int_0^1sh\,d\lambda\qquad\mbox{for every }h\in {L_1[0,1]}.
$$
In particular, if $x^*\in{X^*}$, then
$$
\lim_n\int_0^1s_n(t)\langle{x^*,g(t)}\rangle\,dt =  \int_0^1s(t)\langle{x^*,g(t)}\rangle\,dt
$$
and so
$$
\lim_n\int_0^1s_n(t)g(t)\,dt =  \int_0^1s(t)g(t)\,dt\in{IS_{G}} \qquad\mbox{weakly in }X.
$$
That proves the required weak compactness of $IS_G$.
\\
Now, since
the family $\mathcal S$ of simple functions from $\mathcal{M}$  is dense in $\mathcal{M}$ with respect to the uniform convergence, we have
\begin{eqnarray}\label{giocaL}
IS_{G}=\overline{\biggl\{(P)\int_0^1s(t)g(t)dt: s\in \mathcal{S}\biggr\}}^w,
\end{eqnarray}
where  $w$ denotes the closure in the weak topology.  But as the set in the parenthesis is convex it is in fact the norm closure. Let us present a short proof of (\ref{giocaL}).
\\
We first observe that Pettis integrability of $g$ implies that $\sup_{\|x^*\|\leq 1}\int_0^1|\langle{x^*,g}\rangle|\,d\lambda=K<\infty$.
Next, let us fix $\varepsilon>0$, and any function $\varphi\in \mathcal{M}$. By the assumption, there exists a simple function $s\in\mathcal{S}$ such that $\|\varphi-s\|_{\infty}\leq {\varepsilon}/{K}$.
\\
If $x^*\in{B_{X^*}}$, then
\begin{eqnarray*}
\biggl| \langle x^*,\int g(t)s(t)dt-\int g(t)\varphi(t)dt \rangle \biggr| &\leq& \int | \langle x^*,g(t) \rangle| \cdot |s(t)-\varphi(t)|dt\leq
\\ &\leq& \|s-\varphi\|_{\infty} \int| \langle x^*,g(t) \rangle|dt\leq\ve.
\end{eqnarray*}
Hence
$$
\left\|\displaystyle{ (P)\int s(t)g(t)dt-(P)\int \varphi(t)g(t)dt} \right\|\leq \varepsilon\,.
$$
Now we shall proceed by proving that $IS_{G}$ is the McShane (Birkhoff) integral of $G$. This means that, for every  $\varepsilon>0$ a (measurable)
gauge $\delta$ can be found such that, as soon as $(I_i,t_i)_{i=1}^n$ is a $\delta$-fine McShane partition of $[0,1]$, then
\begin{eqnarray}\label{e5}
d_H(\sum_{i=1}^n G(t_i)\lambda(I_i),IS_{G})\leq \varepsilon.
\end{eqnarray}
So, fix $\varepsilon>0$. Since $g$ is McShane (Birkhoff) integrable, there exists a (measurable) gauge $\delta$ such that, as soon as $(A_i,t_i)_{i=1}^n$ is a generalized $\delta$-fine McShane partition of $[0,1]$, then
\begin{eqnarray}\label{gauge-g}
\biggl\|\sum_{i=1}^ng(t_i)\lambda(A_i)-\int_0^1 g(t)dt\biggr\|\leq \varepsilon.
\end{eqnarray}
Moreover, thanks to the well known Henstock Lemma for the McShane integral (see e.g.\cite[Lemma 2B]{f1995}), it is also possible, for the same partitions, to obtain
\begin{eqnarray}\label{henstocklm}
\biggl\|\sum_{j\in F}\biggl[g(t_j)\lambda(A_j)-\int_{A_j}g(t)dt\biggr]\biggr\|\leq \varepsilon,\end{eqnarray}
whenever $F$ is any finite subset of $\{1,...,n\}$.
So let $(I_i,t_i)_{i=1}^n$ be a $\delta$-fine  McShane partition of $[0,1]$. Let us evaluate
$d_H\biggl(\sum_{i=1}^n G(t_i)\lambda(I_i),IS_{G} \biggr).$
Due to (\ref{giocaL}), we can write
\begin{eqnarray*}
d_H\biggl(IS_{G},\sum_{i=1}^n G(t_i)\lambda(I_i)\biggr)&=& d_H\biggl(\biggl\{\int_0^1\varphi(t)g(t)dt: \varphi \in \mathcal{S}\biggr\},\sum_{i=1}^n G(t_i)\lambda(I_i)\biggr).
\end{eqnarray*}
In order to obtain (\ref{e5}), we will prove that for every
$x\in\sum_{i=1}^n G(t_i)\lambda(I_i)$
there exists $y\in{IS_{G}}$ with $\|x-y\|\leq\ve$ and conversely, given an arbitrary $y\in IS_{G}$, there exists
$x\in\sum_{i=1}^n G(t_i)\lambda(I_i)$ with $\|x-y\|\leq\ve$.
Let
$\sum_{i=1}^na_i g(t_i)\lambda(I_i)$ be a point of $\sum_{i=1}^n G(t_i)\lambda(I_i)$.
We are looking for a proper $\varphi\in{\mathcal S}$. Let
$\varphi:=\sum_{i=1}^na_i1_{I_i}$. \\
Now,  observe that  the mapping $K:[0,1]^n\to [0,+\infty)$, defined as
$$K(a_1,...,a_n)=\biggl\|\sum_{i=1}^na_i\biggl[g(t_i)\lambda(I_i)- \int_{I_i} g(t)dt\biggr]\biggr\|$$
is convex, and therefore it attains its maximum in one of the extreme points of  its domain (bang-bang principle, see e.g. \cite[Corollary 32.3.4]{Ro}):
in other words, there exists a finite subset
$F\subset \{1,...,n\}$ such that
$\max_{(a_i)}K(a_1,...,a_n)=K(e_1,...,e_n)$,
where $e_i=1$ if $i\in F$ and $e_i=0$ otherwise.
Then we have
\begin{eqnarray*}\lefteqn{
\biggl\|\sum_{i=1}^n \left[ a_i g(t_i)\lambda(I_i)-
\int_{I_i} \varphi(t) g(t) dt \right]
\biggr\|}\\
&\leq &
K(e_1,...,e_n)=
\biggl\|\sum_{j\in F}\biggl[g(t_j)\lambda(I_j)-\int_{I_j}g(t)dt\biggr]\biggr\|\leq \varepsilon
\end{eqnarray*}
in virtue of (\ref{henstocklm}).\\
We now turn to the other 
inequality.
We will prove that given
$\varphi \in{\mathcal {S}}$, there exists a point
$x\in \sum_{i=1}^n G(t_i)\lambda(I_i)$ with
$\|x-\int \varphi(t) g(t) dt\| \leq \ve$.
In order to find the proper point, we shall associate with every simple function $\varphi\in \mathcal{S}$,
$\varphi:=\sum_{j=1}^m\varphi_j1_{E_j}$,
and the $\delta$-fine McShane partition $\{(I_1,t_1),\ldots,(I_n,t_n)\}$,
the generalized $\delta$-fine McShane partition  $(I_i\cap E_j,t_i)_{i,j}$,  for every $j=1,...m$.
\\
Notice now that
$G(t_i)\lambda(I_i)=\sum_{j=1}^m G(t_i)\lambda(I_i\cap E_j)$,
for every $i\leq n$.  So given $\varphi=\sum_{j=1}^m\varphi_j1_{E_j}\in \mathcal{S}$, we need a point $ \sum_{i,j}a_{i,j}g(t_i)\lambda(I_i\cap E_j)\in\sum_{i,j}G(t_i)\lambda(I_i\cap E_j)$
such that
$$\biggl\|\sum_{i,j}a_{i,j}g(t_i)\lambda(I_i\cap E_j)-\sum_{i,j}\varphi_j\int_{I_i\cap E_j}g(t)\,dt\biggr\|\leq \ve\,.$$
Let us take $a_{i,j}=\varphi_j$.
We have to evaluate the number
$$\biggl\|\sum_{i,j}\varphi_jg(t_i)\lambda(I_i\cap E_j)-\sum_{i,j}\varphi_j\int_{I_i\cap E_j}g(t)\,dt\biggr\|\,.$$
Let
\begin{eqnarray*}
L(b_1,...,b_{m})=\biggl\|\sum_{i=1}^n\sum_{j=1}^mb_j\biggl[g(t_i)\lambda(I_i\cap E_j)- \int_{I_i\cap E_j} g(t)dt\biggr]\biggr\|
\end{eqnarray*}
be defined for $0\leq{b_j}\leq 1\,, j=1,\ldots,m$.
Applying once again the bang-bang principle and (\ref{henstocklm}), we have for a set
$F\subset \{1,\ldots,m\}$
\begin{eqnarray*}
\sup_{(b_j)}L(b_1,...,b_{m})\leq L(e_1,...,e_{m})
=
\biggl\|\sum_i\sum_{j\in F}\biggl[g(t_i)\lambda(I_i\cap{E_j})-\int_{I_i\cap{E_j}}g(t)dt\biggr]\biggr\|\leq \varepsilon
\end{eqnarray*}
since the partition $(I_i\cap E_j,t_i)_{i,j}$ is $\delta$-fine and so we may apply (\ref{henstocklm}).\\
Thus,
$$\biggl\|\sum_{i,j}\varphi_jg(t_i)\lambda(I_i\cap E_j)-\sum_{i,j}\varphi_j\int_{I_i\cap E_j}g(t)\,dt\biggr\|\leq \ve$$
and so, finally,
$$d_H\biggl(\sum_{i=1}^n G(t_i)\lambda(I_i),IS_{G} \biggr)\leq \varepsilon$$
for all $\delta$-fine McShane decompositions. It follows that $G$ is McShane (Birkhoff) integrable, with integral $IS_{G}$.
\end{proof}


\begin{rem}\rm Using the above result we can conclude that in general Henstock integrability of $g$ does not imply  Henstock integrability of $G$ generated by $g$. In fact, let $g$ be a Henstock but not McShane integrable function. If, by contradiction, $G$ is Henstock integrable then, by \cite[Proposition 3.1]{cdpms2016b}, $G$ is McShane integrable and then, by Theorem \ref{glim}, $g$ is McShane integrable.
While, for the converse, since $G$ is Henstock and $0 \in G(t)$, then $G$ is McShane integrable, by \cite[Corollary 3.2]{dp}, and so it is possible to apply Theorem \ref{glim}.
\end{rem}

  We do not know if a variationally Henstock and Pettis integrable function determines a variationally Henstock integrable multifunction. We have only the following partial result:

\begin{prop}\label{partial}
	Let $g$ be a Pettis and variationally Henstock integrable function of the form:
	$g(t)=\sum_{n=1}^{\infty}x_n 1_{E_n}(t),$
	where $E_n \subseteq  I_n:=(a_{n+1}, a_n)$ for every $n$, with $a_1=1$ and $\lim_{n \to \infty} a_n \downarrow 0$.
	Then $G$ determined by $g$ is variationally Henstock integrable.
\end{prop}
\begin{proof}
	By Fremlin  \cite[Theorem 8]{f1994} we know that $g$ is McShane integrable and so by  Theorem \ref{glim} $G$ is McShane integrable.
	For each $x^*$ we have $s(x^*,G(t))=(x^*g)^+(t)$    and so, for each $t\in [0,1]$:
	$$i(G(t)):=x^*\mapsto \sum_n \langle x^*,x_n\rangle^+1_{E_n}(t)\in l^{\infty}(B_{X^*}),$$
	since $G$ is McShane if and only if
	$i\circ{G}$ is McShane integrable. The McShane integrability of $i\circ{G}$ implies its  Pettis integrability in  $l^{\infty}(B_{X^*})$.
	So we consider now the sequence $(y_n)_n \in  l^{\infty}(B_{X^*})$ given by:
	$y_n(x^*) :=  \langle x^*,x_n\rangle^+$ and we apply \cite[Proposition 4.1]{dpmm} to $i\circ{G}$.
	Using this theorem $i\circ{G}$ is vH-integrable and then, as said in the consequences of  \cite[Theorem 2.4]{cdpms2016}, $G$ is vH-integrable.
\end{proof}

Below we present a few examples of multifunctions generated by functions. 
\begin{exa}\label{ex10}
\rm
	Let $X$ be a Banach space such that McShane and Pettis integrability of $X$-valued functions are not equivalent  (this happens, in general, in non separable Banach space $X$, cf. \cite{FM,r2008,DP,f1994,f1995}). Then, there exists a Pettis integrable multifunction $\vG\colon [0,1]\to{ck(X)}$ such that $0\in{\vG(t)}$ for every $t\in[0,1]$, but $\vG$ is neither Henstock nor McShane integrable.
\end{exa}
\begin{proof}
	Let $g:[0,1]\to{X}$ be Pettis integrable but not McShane integrable. Set $\vG(t):={\rm conv} \{0,g(t)\}$. Then $\vG$ is Pettis integrable. If $\vG$ were Henstock integrable, then by \cite[Proposition 3.1]{cdpms2016}, $\vG$ would be McShane integrable. But then from Theorem \ref{glim} 
	  follows McShane integrability of $g$. A contradiction.
\end{proof}
\begin{exa}\label{ex11}
\rm
	Let $X$ be a Banach space such that McShane and Birkhoff integrability of $X$-valued functions are not equivalent,  as an example the space $l_{\infty} ([0,1])$ can be considered (cf. \cite{r2009}), while the two integrations  are  equivalent, for example, in  Banach  spaces  with  weak$^{\star}$
separable dual unit ball (cf. \cite{BS2011}). Then, there exists a McShane integrable multifunction $\vG:[0,1]\to{ck(X)}$ such that $0\in{\vG(t)}$ for every $t\in[0,1]$, but $\vG$ is neither $\mcH$ integrable (i.e. the version of the Henstock integral when only measurable gauges are allowed) nor Birkhoff integrable.
\end{exa}
\begin{proof}
	We take in Example \ref{ex10} a function $g$ that is McShane but not Birkhoff integrable and follow the same path.
\end{proof}
\begin{exa}{\rm let $X=\ell_2([0,1])$ and let $\{e_t:t\in(0,1]\}$ be its orthonormal system. If $G(t):=\mbox{conv} \{0,{e_t}/t\}$, then $s(x,G)=0$ a.e. for each separate $x\in\ell_2[0,1]$ and so the Pettis integral is equal to zero.
But $G$ is  not Henstock integrable. It is enough to show that $g$ is not Henstock integrable.
So let $\delta$ be any gauge and $\{(I_1,t_1),\ldots,(I_n,t_n)\}$ be a $\delta$-fine Perron partition of $[0,1]$.
 Assume that $0\in{I_1}$, then $t_1\leq|I_1|$. Hence 
$|I_1|/t_1 \geq1$ for $t_1>0$ and so
 $$
  \biggl\| \, \sum_{i\leq{n}}\dfrac{e_i}{t_i} |I_i|\, \biggr\|\geq1\,.
 $$
 Consider now the multifunction given by $H(t):=\mbox{conv}\{0,e_t\}$, where $X$ is as above. We are going to prove that $H$ is Birkhoff-integrable (hence also McShane). Given $\ve>0$, let $n\in\N$ be such that 
$1/\sqrt{n}<\ve$ and $\delta$ be any measurable gauge, pointwise less than $1/n$. If $\{(I_1,t_1),\ldots,(I_m,t_m)\}$ is a $\delta$-fine  partition of $[0,1]$ and $\{J_1,\ldots,J_n\}$ is the division of $[0,1]$ into closed intervals of the same length, then
\begin{eqnarray*}
 \biggl\|\, \sum_{i\leq{m}}e_i|I_i| \, \biggr\| &=& \biggl\|\, \sum_{i\leq{m}}\sum_{k\leq{n}}e_i|I_i\cap{J_k}| \, \biggr\|=\biggl\|\, \sum_{k\leq{n}}\sum_{i\leq{m}}e_i|I_i\cap{J_k}|\, \biggr\|\\
 &=&
  \biggl(\sum_{k\leq{n}}\sum_{i\leq{m}}|I_i\cap{J_k}|^2\biggr)^{1/2}\leq \frac{1}{\sqrt{n}}<\ve\,.
\end{eqnarray*}
(We apply here the inequality $\sum_{i}a_i^2\leq (\sum_{
}a_i)^2$. For each fixed $k\leq{n}$ we take as $a_i$ the number $|I_i\cap{J_k}|$.
)}
\end{exa}

The subsequent  example can be used in order to construct multifunctions that are integrable in one way but not in another one.
\begin{exa}\label{ex1}
	\rm Let $X$ be an arbitrary  Banach space and let $W\in{cb(X)}$ be an uncountable set containing zero.
	Let $f\colon[0,1]\to{X}$ be a scalarly DP--integrable function and let $r\colon[0,1]\to(0,\infty)$
	be a Lebesgue integrable function. Define $\vG\colon[0,1]\to{cb(X)}$ by
	$\vG(t)\colon=r(t)W+f(t)$. One can easily check that $s(x^*,\vG(t))=x^*f(t)+r(t)\sup_{x\in{W}}x^*(x)$. It follows that
	$\vG$ is scalarly DP--integrable.
	
	If we assume that $f$ is DP--integrable, then $\vG \in \mathbb{DP}(cb(X))$ and the decomposition has the following form (with $G(t)=r(t)\,W$):
	$$
	(DP)\int_I\vG=W\int_Ir\,d\lambda+(DP)\int_If\qquad\mbox{for
		every}\quad I\in\mcI.
	$$
	One may replace DP by HKP,  
	H, $\mcH$ and vH.
	For example taking as $f$ a function that is DP but not HKP integrable
	(consider  e.g.
the approximate derivative of  \cite[Example 6.20(c)]{Gor}),
	HKP but not
	H
	(it is enough to take a scalar function $f \in H \setminus L_1$)
	and HKP
	but not Pettis integrable (cf. \cite{gm}), 
	we obtain nontrivial examples of multifunctions integrable in different ways in $cb(X),\,cwk(X)$ or $ck(X)$, depending on the set $W$. DP may be also replaced by Pettis, McShane, Birkhoff or Bochner. In such a case $I$ may be replaced by $E\in\mcL$.
\end{exa}

\begin{que}\label{q1}{\rm Does there exist a positive Henstock integrable multifunction $\vG:[0,1]\to{cb(c_0)}$ that is Pettis integrable but not strongly (i.e. its primitive is not an $h$-multimeasure)?\\
		Does there exist a  Henstock but not McShane integrable multifunction $\vG:[0,1]\to{cb(c_0)}$ possessing a McShane integrable selection?	If in Example \ref{ex1} the function $f$ is strongly measurable and Henstock but not McShane integrable, then $\vG$ does not have any McShane integrable selection.}
\end{que}
\section*{Acknowledgments}
This is a preprint version of an article published in International Journal of Approximate Reasoning. The final authenticated version is available online at:\\ https://www.sciencedirect.com/science/article/abs/pii/S0888613X19301562?via\%3Dihub\\

\noindent{\em \small  No one dies on Earth, as long as he lives in the heart of those who remain; Domenico Candeloro: \dag \,  May, 3, 2019}


\begin{thebibliography}{99}
\bibitem{am}
Agahi, H., Mesiar, R.: 
On Choquet-Pettis expectation of Banach-valued functions: a counter example,
Internat. J. Uncertain. Fuzziness Knowledge-Based Systems {\bf 26} (2),  (2018), 
255-259.

\bibitem{BP}  Balcerzak, M.,  Potyra{\l}a, M.:  Convergence theorems for the Birkhoff integral, Czech. Math. J. {\bf 58}, (2008), 1207-1219

\bibitem{bcs2015} Boccuto, A.,   Candeloro, D.,   Sambucini,  A.R.: Henstock multivalued integrability in Banach
 lattices with respect to pointwise non atomic measures, Atti Accad. Naz. Lincei Rend. Lincei Mat. Appl.
 {\bf 26} (4), (2015),  363--383 Doi: 10.4171/RLM/710

\bibitem{bms}  Boccuto, A., A.M. Minotti, A.M.,  Sambucini, A. R.:  Set-valued Kurzweil-Henstock integral in Riesz space setting,   PanAmerican Mathematical Journal {\bf 23} (1), (2013), 57--74.

 \bibitem{BS2011}  Boccuto, A.,  Sambucini, A. R.: A note on comparison between Birkhoff and McShane-type integrals for multifunctions, Real Anal. Exchange  \textbf{37} (2),  (2012), 315-324.

\bibitem{ccgs}
 Candeloro, D.,   Croitoru, A.,  Gavrilut A.,  Sambucini,  A.R.: An extension of the Birkhoff integrability for multifunctions, Mediterranean J. Math. {\bf 13} (5),   (2016), 2551-2575, Doi: 10.1007/s00009-015-0639-7

\bibitem{cdpms2016}
Candeloro, D.,  Di Piazza, L.,   Musia{\l}, K.,  Sambucini,  A.R.: Gauge integrals and selections of weakly compact valued multifunctions, J. Math. Anal. Appl.  {\bf 441} (1), (2016),  293--308, Doi: 10.1016/j.jmaa.2016.04.009

\bibitem{cdpms2016b}
Candeloro, D.,  Di Piazza, L.,   Musia{\l}, K., Sambucini,  A.R.: Relations among gauge and Pettis integrals
 for multifunctions with weakly compact convex values, Annali di Matematica
  {\bf 197} (1), (2018), 171-183. Doi: 10.1007/s10231-017-0674-z

\bibitem{cdpms2016a}
Candeloro, D., Di Piazza, L.,   Musia{\l}, K., Sambucini,  A.R.: Some new results on integration for multifunction,  Ricerche di Matematica {\bf 67} (2), (2018), 361-372. Doi: 10.1007/s11587-018-0376-x

\bibitem{monat2019}
Candeloro, D.,  Di Piazza, L.,   Musia{\l}, K., Sambucini,  A.R.: Integration of
  multifunctions with closed  convex values in arbitrary Banach spaces,  submitted (2018), arxiv 1812.00597


\bibitem{CMeS}  Candeloro,  D.,  Mesiar, R.,   Sambucini, A.R.: A special class of fuzzy measures: Choquet integral and applications,  Fuzzy Sets and Systems, {\bf 355}, (2019), 83-99,  Doi: 10.1016/j.fss.2018.04.008

\bibitem{cs2015} Candeloro, D., Sambucini,  A.R.: Comparison between some norm and order gauge integrals in Banach lattices, PanAm. Math. J.  {\bf 25} (3), (2015),  1-16

\bibitem{ncm}  Caponetti, D.,   Marraffa, V.,   Naralenkov, K.: On the integration of Riemann-measurable vector-valued functions, Monatsh. Math. {\bf 182},  (2017), 513-536, doi:10.1007/s00605-016-0923-z

\bibitem{capot}
 Capotorti, A., Coletti, G., Vantaggi, B.:
Standard and nonstandard representability of positive uncertainty orderings, 
Kybernetika, {\bf 50} (2), (2014),  189-215

\bibitem{ckr1} Cascales, B., Kadets,  V.,  Rodr\'{i}guez, J.: Measurability and selections of multifunctions
in Banach spaces, J. Convex Analysis {\bf 17}, (2010) 229-240

\bibitem{CV} Castaing, C.,   Valadier, M.: Convex Analysis and Measurable Multifunctions,
 Lecture Notes Math. 580, Springer-Verlag, Berlin-New York (1977)

\bibitem{ci1} Cicho\'{n}, M.,  Cicho\'{n}, K.,   Satco, B.: Differential inclusions and multivalued integrals,
Discuss. Math. Differ. Incl. Control Optim. {\bf 33} (2),  (2013),  171-191

\bibitem{coletti0}
Coletti, G., Petturiti, D., Vantaggi, B.: Interval-based possibilistic logic in a coherent setting, Lecture Notes in Computer Science   10351, (2017), 75-84

\bibitem{coletti}
Coletti, G., Petturiti, D., Vantaggi, B.: Models for pessimistic or optimistic decisions under different uncertain scenarios,
International Journal of Approximate Reasoning  {\bf 105}, (2019), 305–326

\bibitem{dp-ma} Di Piazza, L.,  Marraffa, V.: The McShane, PU and Henstock integrals of Banach valued functions,
Czechoslovak Math. J. {\bf 52} (3), (2002), 609--633

\bibitem{dpmm} Di Piazza, L.,  Marraffa, V.,  Musia{\l}, K.: Variational Henstock integrability of Banach space
 valued function, Math. Bohem. {\bf 141} (29),  (2016), 287--296
Doi: 10.21136/MB.2016.19

\bibitem{dms} Di Piazza L., Marraffa V., Satco B.:  Closure properties for integral problems driven by regulated functions via convergence results. J. Math. Anal. Appl. 466 (2018), no. 1, 690–710.

\bibitem{dms1} Di Piazza L., Marraffa V., Satco B.:  1.	Approximating the solutions of differential inclusions driven by measures,  Ann. Mat. Pura Appl. (2019) Doi: 10.1007/s10231-019-00857-6


\bibitem{dp} Di Piazza, L., Musia{\l}, K.: Relations among Henstock, McShane and Pettis integrals for multifunctions with compact convex values,  Monatsh. Math. {\bf 173}, (2014), 459--470

\bibitem{DP} Di Piazza,	L., Preiss, D.: When do McShane and Pettis integrals coincide?,   Illinois J. Math. 47 (4), (2003), pp. 1177--1187, ISSN: 0019-2082. 

\bibitem{FM} Fremlin,  D. H., Mendoza, J.:   On the integration of vector-valued functions, Illinois J. Math.
 {\bf 38}, (1994), 127--147

\bibitem{f1994}  Fremlin,  D. H.:  The Henstock and McShane integrals of vector-valued functions,
Illinois J. Math. {\bf 38} (3), (1994), 471--479

\bibitem{f1995}  Fremlin,  D. H.   The generalized McShane integral,  Illinois J. Math. {\bf 39} (1), (1995), 39--67.

\bibitem{gav2014} Gavrilut, A.: Remarks on monotone interval-valued set multifunctions, Information Sciences, {\bf 259}, (2014), 225-230

\bibitem{gm} G\'amez, J.L.,   Mendoza, J.:  On Denjoy-Dunford and Denjoy-Pettis integrals,  Studia Math. {\bf 130}, (1998), 115--133

\bibitem{Gor}   Gordon, R.A.: The Integrals of Lebesgue, Denjoy, Perron and Henstock, Grad. Stud. Math. 4,
 AMS, Providence (1994)

\bibitem{kal}
Kaliaj, S. B.:
The new extensions of the Henstock-Kurzweil and the McShane integrals of vector-valued functions, 
Mediterr. J. Math. {\bf 15} (1), (2018),  Art. 22, 16 pp. 

\bibitem{labu}  Labuschagne, C.C.A.,  Pinchuck, A.L.,  van Alten, C.J.: \emph{A vector lattice version of R{\aa}dstr\"{o}m's embedding theorem}, Quaest. Math. \textbf{30} (3), (2007), 285-308.

\bibitem{latorre}  La Torre, D.,Mendivil, F.: Minkowski-additive multimeasures, monotonicity and self-similarity,
Image Analysis and Stereology, {\bf 30} (3), (2011), 135-142
doi:10.5566/ias.v30.p135-142

\bibitem{mu15} Musia{\l}, K.: Pettis integrability of multifunctions with values in arbitrary Banach
spaces, J. Convex Analysis {\bf 18}, (2011), 769-810

\bibitem{nara}
Naralenkov, K.M.: A Lusin type measurability property for vector-valued functions,
 J. Math. Anal. Appl.  {\bf 417} (1),  (2014), 293-307

\bibitem{pap} Pap, E.:
Multivalued functions integration: from additive to arbitrary non-negative set function,  On logical, algebraic, and probabilistic aspects of fuzzy set theory,
Stud. Fuzziness Soft Comput., {\bf 336}, (2016), 257-274, Springer.

\bibitem{pap-iran} Pap, E. Iosif, A., Gavrilut, A.:
Integrability of an interval-valued multifunction with respect to an interval-valued set multifunction, Iranian J. of Fuzzy Systems {\bf 15} (3), (2018), 47-63.

\bibitem{park}
Park, C. K.: Set-valued Choquet Pettis integrals, Korean J. Math., {\bf 20} (4), (2012), 381-393.

\bibitem{pva} Petturiti, D., Vantaggi, B.:
Upper and lower conditional probabilities induced by a multivalued mapping, J. of Math. Anal. and Appl., {\bf 458} (2), (2018), 1214-1235

\bibitem{Ro}  Rockafellar, R.: Convex Analysis, Princeton, New Jersey. Princeton Univ. Press (1970)

\bibitem{r2008} Rodr\'{i}guez, J.: On the equivalence of McShane and Pettis integrability in non-separable Banach spaces, J. Math. Anal. Appl. {\bf 341}, (2008), 80-90


\bibitem{r2009} Rodr\'{i}guez, J.:
Some examples in vector integration, Bull. Austr. Math. Soc., {\bf 80}, (2009), 384-392.

\bibitem{ar}  Sambucini, A. R.: The Choquet integral with respect to fuzzy measures and applications, Math. Slov. {\bf 67} (6), (2017),  1427-1450, Doi: 10.1515/ms-2017-0049.


\bibitem{sun} Sun, L., Dong, H., Liu, A. X.: Aggregation Functions Considering Criteria Interrelationships in Fuzzy Multi-Criteria Decision Making: State-of-the-Art, IEEE Access, 
{\bf  6}, (2018),  68104-68136,
Doi:  10.1109/ACCESS.2018.2879741
\end{thebibliography}
\end{document}